\documentclass{amsart}
\usepackage{amssymb,latexsym,graphics,enumerate,color}
\usepackage{fullpage}
\numberwithin{equation}{section}

\usepackage[normalem]{ulem}

\usepackage{xcolor,cancel}

\newtheorem{thm}{Theorem}[section]

\newtheorem{lem}[thm]{Lemma}
\newtheorem{prop}[thm]{Proposition}

\theoremstyle{remark}

\theoremstyle{remark}

\theoremstyle{definition}

\newcommand{\la}{\langle}
\newcommand{\ra}{\rangle}

\renewcommand{\O}{{\mathcal{O}}}
\renewcommand{\S}{{\mathbb{S}}}

\newcommand{\R}{{\mathbb{R}}}

\newcommand{\cd}{{\,\cdot\,}}

\newcommand{\ang}{{\not\negmedspace\nabla}}
\newcommand{\angdelta}{{\not\negthickspace\Delta}}

\newcommand{\supp}{{\text{supp }}}

\newcommand{\g}{{\mathfrak{g}}}

\begin{document}
\bibliographystyle{plain}

\title
{Localized energy for wave equations with degenerate trapping}

\author{Robert Booth}
\author{Hans Christianson}
\author{Jason Metcalfe}
\author{Jacob Perry}
\email{rjbooth@live.unc.edu, hans@math.unc.edu, metcalfe@email.unc.edu,\newline perryja@live.unc.edu}
\address{Department of Mathematics, University of North Carolina, Chapel Hill, NC 27599-3250}

\thanks{Booth and Perry were supported, in part, by NSF grants
  DMS-1054289 (PI Metcalfe), DMS-1500817 (PI Taylor), and DMS-1352353
  (PI Marzuola).  Christianson was supported by NSF
grant DMS-1500812 and Metcalfe
by NSF grant DMS-1054289.}



\begin{abstract}
Localized energy estimates have become a fundamental tool when
studying wave equations in the presence of asymptotically flat
background geometry.  Trapped rays necessitate a loss when compared to
the estimate on Minkowski space.  A loss of regularity is a common
way to incorporate such.  When trapping is sufficiently weak, a
logarithmic loss of regularity suffices.  Here, by studying a warped
product manifold introduced by Christianson and Wunsch, we encounter
the first explicit example of a situation where an estimate with an
algebraic loss of regularity exists and this loss is sharp.  Due to
the global-in-time nature of the estimate for the wave equation, the
situation is more complicated than for the Schr\"odinger equation.
An initial estimate with sub-optimal loss is first obtained, where
extra care is required due to the low frequency contributions. 
 An improved estimate is then established using energy
functionals that are inspired by WKB analysis.  Finally, it is shown
that the loss cannot be improved by any power by saturating the
estimate with a quasimode.
\end{abstract}

\maketitle

\section{Introduction}
When studying wave equations on asymptotically flat backgrounds,
(integrated) local energy estimates have become a fundamental tool.
In fact, in a number of scenarios, it has been shown that these local
energy estimates imply other known measures of dispersion, such as
pointwise decay estimates \cite{Tat_Ldecay}, \cite{MTT_Price} and
Strichartz estimates \cite{MT_para}. 
When there are null geodesics that remain in a compact set for all
times, trapping is said to occur, and trapping is a known obstruction
to local energy estimates \cite{ralston}, \cite{sbierski}.  So any estimates in the presence of
trapping must have a loss when compared to those estimates available
on Minkowski space.  This loss is often realized as a loss of regularity.
In many situations where the trapping is sufficiently hyperbolic (unstable), an estimate
with a minimal loss (say, a logarithmic loss of regularity) can be
recovered.  See, e.g., \cite{Burq_Zwor_Geom}, \cite{C_disp},
\cite{CdV_P}, \cite{Ik_two, Ik_several}, \cite{MMTT_Str_BH},
\cite{NonnZw}, \cite{TT}, \cite{WZwo}.  On the other hand, in presence
of elliptic (stable) trapped rays, nearly everything, say everything
but a logarithmic amount of decay, is lost \cite{Burq_Fr}.  Explicit examples
where an algebraic loss of regularity is necessary and sufficient to
prove integrated local energy estimates have not previously appeared.

A similar story exists for the Schr\"odinger equation, where the
analog of the local energy estimate is the local smoothing estimate
\cite{Con_Sau}, \cite{Sjolin}, \cite{Vega}.  There, as the speed of propagation for the
Schr\"odinger equation is proportional to the frequency, the estimate
gives a $1/2$-degree of smoothing.  When trapped rays exist, it is
known that the full $1/2$-degree of smoothing cannot be recovered
\cite{Doi_Riem}.  And as above, until recently, examples where the trapping
caused a minimal loss and examples where the trapping disallowed all
but a minimal amount of the smoothing were known, but nothing
explicitly had been established in between.

In \cite{CW}, the authors considered the Schr\"odinger equation in the
presence of degenerate trapping on a product manifold or a surface of
revolution.  In this example, when $2m$ denotes
the degeneracy of the trapping, they succeeded in
showing that a loss of $\frac{m-1}{2(m+1)}$ derivatives when compared to the
estimate on Euclidean space was necessary and sufficient.  This
provided the first example of local smoothing with a sharp algebraic
loss.  Subsequent studies include \cite{CM}, \cite{C_hfres}, \cite{C_Str} which
address, respectively, the loss caused
by inflection points in the generating function, the case of
infinitely degenerate trapping, and Strichartz estimates
with degenerate trapping.

Here we consider the wave equation on the same
geometric background.  The local smoothing estimate is interesting
even locally in time, and in fact, \cite{CW} proved such for times in
the unit interval.  For the wave equation, however, local in time
estimates follow trivially from uniform energy bounds.  Thus for
the wave equation only global in time estimates are considered, and this
brings some new low frequency analysis into play.

We shall now describe the geometric setting.  We consider the manifold
$\R\times \R \times \S^2$ equipped with the Lorentzian metric
\[ds^2 = -dt^2 + dx^2 + a(x)^2\,d\sigma^2_{\S^2}.\]
The generating function of this surface of revolution is given by $a(x)=(x^{2m}+1)^{1/2m}$.
This constructs a two ended surface, which is asymptotically Minkowski
in both $x$-directions.  Due to the critical point of $a(x)$ at $x=0$,
a surface of trapped null geodesics is formed.  The case of $m=1$ is
the nondegenerate case, which is the well-studied case of hyperbolic
trapping mentioned above, and is a simplified
model for the trapping that occurs on, e.g., Schwarzschild
spaces.  The analysis of \cite{MMTT_Str_BH} can be directly mimicked to
provide the local energy estimates with minimal loss.  It is the
degenerate cases $m\ge 2$ that interest us here.

In this geometric setup and in these coordinates, we note that
\begin{equation}
  \label{box}
  \Box_\g u = - \partial_t^2 u +
  a(x)^{-2}\partial_x\Bigl[a(x)^2\partial_x u\Bigr] + \frac{1}{a(x)^2}\angdelta_{\S^2}u
\end{equation}
where $\angdelta_{\S^2}$ denotes the Laplacian on $\S^2$.
Using the product structure of the metric, we will separate space and
time in the volume form and indicate $dV = a(x)^2\,dx\,d\sigma_{\S^2}$.
We will use $dV\,dt$ when the volume form of the full space-time is desired.

As this metric is static, there is a natural coercive
energy, which is conserved when $\Box_g u = 0$:
\[E[u](t) = \int (\partial_t u)^2 + (\partial_x u)^2 +
\frac{1}{a(x)^2}|\ang_0 u|^2\,dV = \int |\partial u|^2\,dV.\]
Here we are using $\partial u = (\partial_t u, \partial_x u, 
\ang_0
u)$, where $\ang_0$ denotes derivatives tangential to $\S^2$, 
so that $|\partial u|^2 = |\partial u|_\g^2 = (\partial_t u)^2 +
(\partial_x u)^2 + \frac{1}{a(x)^2}|\ang_0 u|^2$.  
The conservation of energy can be proved by multiplying $\Box_\g u$ by
$\partial_t u$ and integrating by parts.  The same method, more
generally, yields:
\begin{equation}
  \label{energy}
  E[u](t) \lesssim E[u](0) + \Bigl|\int_0^t\int \Box_\g u \partial_t u\,dV\,dt\Bigr|.
\end{equation}

We now describe the spaces that shall be used to measure the local
energy:
\[\|u\|_{LE} = \sup_{j\ge 0} 2^{-j/2}\|u\|_{L^2L^2([0,T] \times \{\la
  x\ra\approx 2^j\})},\quad \|u\|_{LE^1} = \|(\partial u, \la
x\ra^{-1} u)\|_{LE}.\]
Forcing terms will be frequently measured in the corresponding dual
norm:
\[\|F\|_{LE^*} = \sum_{j\ge 0} 2^{j/2} \|F\|_{L^2L^2([0,T]\times \{\la
  x\ra\approx 2^j\})}.\]
Here $L^2L^2$ indicates the full space-time $L^2$ norm (where to mimic
what is commonly seen on Minkowski space, the first $L^2$ is in $t$
and the second $L^2$ is over the spatial variables
$(x,\theta,\phi)$).  The norms in $t$ will be taken over $[0,T]$, but
all constants will be independent of $T$, which yields the desired
global estimates.  
We shall use notations such as $LE_R$ to indicate
the $LE$ norm restricted to a single dyadic annulus with $2^j\approx
R$ and $LE_{>R}$ to indicate the $LE$ norm with the restriction that
$2^j > R$.

On $(1+3)$-dimensional Minkowski space, the uniform energy bound and the (integrated)
local energy estimate read as
\[\|\partial u\|_{L^\infty L^2} + \|u\|_{LE^1} \lesssim \|\partial
u(0,\cd)\|_{L^2} + \|\Box u\|_{L^1L^2 + LE^*}.\]
Such estimates originated in the works \cite{Mora1, Mora2, Mora3} and
can be proved by pairing $\Box u$ with $C\partial_t u + \frac{r}{r+2^j} \partial_r u +
\frac{1}{r+2^j} u$, integrating over a space-time slab, integrating by
parts, and using a Hardy inequality.  See, e.g., \cite{St},
\cite{MS_SIAM}.  And in fact, these estimates hold on any 
stationary, Lorentzian, asymptotically flat spaces provided that there are no
trapped rays and there are no eigenvalues nor resonances on the real
line or in the lower half of the complex plane.  See \cite{MST}, which
can also be referred to for a more complete history of such estimates.

As was done in preceding works such as \cite{CW}, \cite{MMTT_Str_BH},
our first goal is to develop an estimate that provides the local
energy estimate away from the trapped set.  The loss here will be
manifest through a coefficient that vanishes where there is trapping.
In \cite{CW}, this follows from a relatively standard integration by
parts argument.  But as we must now consider global-in-time
estimates, some new low frequency
contributions must be considered.  However, using a refinement of the
exterior estimate of \cite{MST}, which is inspired by that of
\cite{MMT_Str_Sch}, we shall obtain:
\begin{thm}\label{thm_lossy}\footnote{For a normed space
    $X=\{x\,:\,\|x\|_X<\infty\}$, we have $aX =
    \{ax\,:\,\|x\|_X<\infty\} = \{y\,:\, \|a^{-1}y\|_X<\infty\}$.  And
  thus, we denote $\|y\|_{aX} = \|a^{-1}y\|_X$.}
On the geometry described above, if, for each $t$, $u(t,x)$ vanishes for
sufficiently large $|x|$, we have
\begin{multline}
  \label{main_lossy}
\sup_t E[u](t)^{1/2} + \|\partial_x u\|_{LE} + \Bigl\|\frac{|x|^m}{\la x\ra^m} \partial_t
 u\Bigr\|_{LE} + \Bigl\|\frac{|x|^m}{\la x\ra^m}
 \frac{1}{a(x)}|\ang_0 u|\Bigr\|_{LE} + \|\la x\ra^{-1} u\|_{LE}
 \\\lesssim E[u](0)^{1/2} + \|\Box_\g u\|_{L^1L^2 + (|x|/\la x\ra)^m LE^*}
\end{multline}
and
\begin{multline}
  \label{main_lossy2}
\sup_t E[u](t)^{1/2} + \|\partial_x u\|_{LE} + \Bigl\|\frac{|x|^m}{\la x\ra^m} \partial_t
 u\Bigr\|_{LE} + \Bigl\|\frac{|x|^m}{\la x\ra^m}
 \frac{1}{a(x)}|\ang_0 u|\Bigr\|_{LE} + \|\la x\ra^{-1} u\|_{LE}
 \\\lesssim E[u](0)^{1/2} + \|\la D_\omega\ra^{\frac{m-1}{2(m+1)}}
   \Box_g u\|^{1/2}_{LE^*} \|\la
   D_\omega\ra^{-\frac{m-1}{2(m+1)}} \partial u\|^{1/2}_{LE} + \|\Box_\g u\|_{LE^*}.
\end{multline}
\end{thm}

We note that when $m=1$ this is the direct analog of \cite[Theorem
1.2]{MMTT_Str_BH}.  See also \cite{Booth_Masters} where the case of
$m=1$ was considered.  Theorem \ref{thm_lossy} is already essential as it provides
lossless local energy estimates away from the trapping at $x=0$ and 
everywhere for $\partial_x u$ and for $\la x\ra^{-1} u$.  Moreover, in
the sequel, this estimate will allow us to localize our analysis to a
small neighborhood of the trapping.

We then seek to improve the corresponding loss.  One option would be
to use the resolvent bounds that were developed in \cite{CW}.  As an
alternative, we shall adapt the techniques of \cite{MMTT_Str_BH},
which are based on energy functionals that are inspired by WKB
theory.  We hope that this alternate method may provide more
flexibility when considering more general cases of degenerate
trapping.  As a brief motivation, see \cite{Toh}, which adapted the
methods of \cite{MMTT_Str_BH} to the setting of Kerr backgrounds with
sufficiently small angular momenta.  Using these methods, we prove
\begin{thm}\label{thm_sharp}
  Let $m\ge 2$.  On the geometry described above, suppose $\Box_\g u =
  0$.  Then we have
  \begin{equation}
    \label{sharp_est}
    \|u\|^2_{LE^1} \lesssim E[\la
    D_\omega\ra^{\frac{m-1}{2(m+1)}} u](0).
  \end{equation}
\end{thm}
Here $D_\omega = -i \sqrt{\frac{1}{a^2} \angdelta_{\S^2}}$.  For
clarity of exposition, we have only stated this for the homogeneous
case.  A forcing term, which is measured in $\la
D_\omega\ra^{-\frac{m-1}{m+1}}LE^*$ may be included.  This is done,
e.g., in \cite{MMTT_Str_BH} for the nondegenerate case, and the
argument here would be similar.

Our last task is to show that this estimate is sharp in the sense that
the estimate fails to hold if the loss is decreased by any power.
Here we will rely on the quasimode that was constructed in \cite{CW}.
Saturating the estimate for the wave equation is a bit different than
saturating the local-in-time estimate for the Schr\"odinger equation.
Rather than saturating the amount of smoothing available, here we must
saturate the integrability.  The spectral parameter that arises here
is squared due to the wave equation being second order in time.  By
choosing the correct root, the solution that is constructed from the
quasimode has appropriate growth to do exactly this.  We, in fact,
prove:
\begin{thm}
  \label{thm_quasi} 
There exist compactly supported functions $\psi_0$ and $\psi_1$ and a $T>0$ so
that the solution $\psi$ to
\[\Box_\g \psi = 0,\quad (\psi,\partial_t\psi)|_{t=0}=(\psi_0,\psi_1)\]
satisfies
\[\int_0^T \|\beta(|x|) a(x)^{-1}|\ang_0 \psi|\|^2_{L^2}\,dt \gtrsim \|\la
D_\omega\ra^{\frac{m-1}{2(m+1)}} a(x)^{-1}|\ang_0 \psi_0|\|_{L^2}^2 + \|\la
D_\omega\ra^{\frac{m-1}{2(m+1)}} \psi_1\|_{L^2}^2.\]
\end{thm}

The article is organized as follows.  In the next section, we shall
prove Theorem \ref{thm_lossy}.  To do so, the analysis is broken into
a low frequency regime and a range of frequencies that is bounded away
from zero.  The next section is devoted to refining the analysis to
prove the sharp estimate as stated in Theorem \ref{thm_sharp}.  The
arguments here are adaptations of those developed in
\cite{MMTT_Str_BH}.  Finally, in the last section, we prove Theorem
\ref{thm_quasi}, which shows that the loss of regularity cannot be
improved by any power.

\section{A lossy estimate - Proof of Theorem \ref{thm_lossy}}

In this section, we shall prove Theorem \ref{thm_lossy}.  Some of this
section is inspired by \cite{MST}, but as we are examining an explicit
metric that is stationary and has a product structure, the methods can
be significantly simplified.  

Throughout this article, we fix $\beta(\rho)$ to be a smooth, monotone cutoff that is 1 for $\rho <1/2$
and $0$ for $\rho >1$.  To prove Theorem \ref{thm_lossy}, we
shall analyze $u_{<\tau} = \beta(D_t/\tau) u$ and $u_{>\tau} =
(1-\beta(D_t/\tau)) u$ separately.  The parameter $0<\tau\ll 1$ will be
chosen later.  In particular, we shall establish the following
estimate that handles any frequency range that is bounded away from $0$.
\begin{prop}\label{HFprop}
  For any $\tau>0$, we have
  \begin{multline}
    \label{HFmain}
\|\partial_x u_{>\tau}\|^2_{LE} + \Bigl\|\frac{|x|^m}{\la x\ra^m} \partial_t
 u_{>\tau}\Bigr\|^2_{LE} + \Bigl\|\frac{|x|^m}{\la x\ra^m}
 \frac{1}{a(x)}|\ang_0 u_{>\tau} |\Bigr\|^2_{LE} + \|\la x\ra^{-1}
 u_{>\tau}\|^2_{LE}
\\ \lesssim E[u](0) + \|\Box_\g u\|_{L^1L^2 + (|x|/\la x\ra)^m
   LE^*}
\|(\partial u, a(x)^{-1} u)\|_{L^\infty L^2\cap (\la x\ra / |x|)^m LE}
 \end{multline}
and
  \begin{multline}
    \label{HFmain2}
\|\partial_x u_{>\tau}\|^2_{LE} + \Bigl\|\frac{|x|^m}{\la x\ra^m} \partial_t
 u_{>\tau}\Bigr\|^2_{LE} + \Bigl\|\frac{|x|^m}{\la x\ra^m}
 \frac{1}{a(x)}|\ang_0 u_{>\tau} |\Bigr\|^2_{LE} + \|\la x\ra^{-1}
 u_{>\tau}\|^2_{LE}
\\ \lesssim E[u](0) + \|\la D_\omega\ra^{\frac{m-1}{2(m+1)}} \Box_\g u\|_{LE^*}
\|\la D_\omega\ra^{-\frac{m-1}{2(m+1)}}\partial u\|_{LE} + \|\Box_\g u\|_{LE^*}^2.
 \end{multline}
\end{prop}
It is here that the effects of trapping are observed, as is evidenced
by the coefficients that vanish at the location of the trapping.

The above will be combined with the following estimate for
sufficiently small frequencies.
\begin{prop}\label{LFprop}
 For any $\tau>0$ sufficiently small, we have
 \begin{equation}
   \label{LFmain}
   \|u_{<\tau}\|^2_{LE^1} \lesssim E[u](0) + \|\Box_\g u\|^2_{L^1L^2 + 
   LE^*}.
 \end{equation}
\end{prop}
The vanishing at the trapping is irrelevant here as trapping is a high
frequency phenomenon.

As the metric is stationary, it is trivial to commute the frequency
cutoff with $\Box_\g$.  Upon combining \eqref{energy}, \eqref{HFmain}, and
\eqref{LFmain}, Theorem \ref{thm_lossy} results.

All of the results in this section come from a multiplier method and
an associated integration by parts.  At this point, we will record the
following abstract calculation that shall be used many times in the
sequel.  Suppose $w, g\in C^2$, $f\in C^1$, and for each $t$, $w(t,x)$ vanishes
for large enough $|x|$.  Then
  \begin{multline}\label{ibp}
    -\int_0^T\int \Box_\g w \Bigl\{f(x) \partial_x w +
g(x) w\Bigr\} \,dV\,dt = \int \partial_t w
    \Bigl(f(x)\partial_x w +
    g(x) w\Bigr)
    dV\Bigl|_0^T
\\+\int_0^T\int \Bigl(f'(x) + g(x) - \frac{1}{2}\Bigl\{a(x)^{-2}\partial_x(a(x)^2
f(x))\Bigr\}\Bigr) (\partial_x w)^2 \,dV\,dt
\\+ \int_0^T\int \Bigl(f(x)\frac{a'(x)}{a(x)} + g(x) - \frac{1}{2}\Bigl\{a(x)^{-2}\partial_x(a(x)^2
f(x))\Bigr\}\Bigr) \frac{1}{a(x)^2}|\ang_0 w|^2\,dV\,dt
\\
+ \int_0^T\int \Bigl(-g(x) + \frac{1}{2}\Bigl\{a(x)^{-2}\partial_x(a(x)^2
f(x))\Bigr\}\Bigr) (\partial_t w)^2\,dV\,dt
\\
-\frac{1}{2}\int_0^T\int \Bigl(a(x)^{-2}\partial_x[a(x)^2\partial_x g]\Bigr)w^2\,dV\,dt.
  \end{multline}

\subsection{Exterior estimates}
We first establish an estimate away from $x=0$, which is where the
trapping occurs.  This estimate shows that the local energy estimates
necessarily hold near the infinite ends with a lower order error term
that is supported on a compact region.  An estimate analogous to this
was first established in \cite{MMT_Str_Sch} for the Schr\"odinger
equation and in \cite{MST} for more general wave equations.  Here the
warped product structure simplifies the choice of multiplier, but some
care must be taken to accommodate having two ends.

Though our proof does not significantly differ from \cite{MST}, we
have provided a sharper statement that allows for a difference in the
radius outside of which you hope to estimate your solution and the
radius at which you are cutting away and permitting an error term.
This difference in radii provides a degree of smallness that allows us
to simplify the low frequency analysis.

\begin{prop}\label{prop.exterior}
For any parameters $R$ and $R_1$ satisfying $(1/2)R\ge R_1\ge 2$, we have the following:
\begin{equation}
  \label{midext}
  \|u\|_{LE^1_{|x|>R}}^2 \lesssim E[u](0) + \int_0^T\int |\Box_\g
  u|\Bigl(|\partial u| + \frac{1}{a(x)}|u|\Bigr)\,dV\,dt + R^{-1}R_1^{-1}\|u\|^2_{LE_{|x|\approx R_1}}.
\end{equation}
\end{prop}

\begin{proof}
We shall use \eqref{ibp} with 
\begin{align*}f(x)&=(1-\beta(|x|/R_1))h(x),\qquad h(x)=\frac{x}{|x|+\rho}, \qquad \rho\ge
R,\\ 
g(x) &= \frac{1}{2}a(x)^{-2} h(x)\partial_x \Bigl[(1-\beta(|x|/R_1))
       a(x)^2\Bigr].
\end{align*}
 We can then compute
\[h'(x) = \frac{\rho}{(|x|+\rho)^2},\quad h''(x) = -\frac{2\rho\,
  \text{sgn}(x)}{(|x|+\rho)^3}.\]

We examine the coefficients of each term in the right side of
\eqref{ibp}.  To start, using that $\beta$ is a monotonically
decreasing function, we have
\begin{equation}\begin{split}\label{extDx}
  f'(x)+g(x) - \frac{1}{2}\Bigl\{a(x)^{-2}\partial_x(a(x)^2
  f(x))\Bigr\} &= \frac{1}{2} (1-\beta(|x|/R_1)) h'(x) - R_1^{-1}
                 \beta'(|x|/R_1) \text{sgn}(x) h(x)\\
&\ge \frac{1}{2}(1-\beta(|x|/R_1)) \frac{\rho}{(|x|+\rho)^2}.
\end{split}
\end{equation}
And moreover, the right side is $\approx 1/\rho$ on $|x|\approx \rho$.

For the angular derivatives, we have
\begin{equation}\label{extAng}
f(x)\frac{a'(x)}{a(x)} + g(x) -
\frac{1}{2}\Bigl\{a(x)^{-2}\partial_x(a(x)^2 f(x))\Bigr\} =
(1-\beta(|x|/R_1)) \Bigl(\frac{a'(x)}{a(x)}h(x) - \frac{1}{2}h'(x)\Bigr).
\end{equation}
We record that
$\frac{a'(x)}{a(x)}h(x) - \frac{1}{2}h'(x) \ge
\frac{1}{2}\frac{|x|}{(|x|+\rho)^2}$ on the support of
$1-\beta(|x|/R_1)$.  And this coefficient is also $\approx 1/\rho$ on
$|x|\approx \rho$.

For the time derivatives, we simply get
\begin{equation}
  \label{extDt}
  -g(x)+\frac{1}{2}\Bigl\{a(x)^{-2}\partial_x(a(x)^2 f(x))\Bigr\} =
 \frac{1}{2} (1-\beta(|x|/R_1)) h'(x),
\end{equation}
which is everywhere non-negative and $\approx 1/\rho$ on $|x|\approx
\rho$.

It remains to examine the lower order term.  We compute
\begin{multline}
  \label{extLot}
  -\frac{1}{4}
\Bigl\{a(x)^{-2}\partial_x\Bigl[a(x)^2\partial_x\Bigl(h(x)a(x)^{-2}\partial_x
(a(x)^2)\Bigr)\Bigr]\Bigr\} \\= - \frac{1}{2}
\frac{x^{2m-1}}{1+x^{2m}}h''(x) + \frac{2m-1}{(1+x^{2m})^2}
  x^{2m-2}\Bigl(\frac{h(x)}{x} - h'(x)\Bigr) +
  \frac{2m-1}{(1+x^{2m})^3}x^{2m-3} m(x^{2m}-1)h(x).
\end{multline}
Here we observe that $-\frac{x^{2m-1}}{1+x^{2m}}h''(x)$ is everywhere
non-negative and is $\gtrsim \rho^{-3}$ when $|x|\approx \rho$.
Moreover, we have that $h(x)/x\ge h'(x)$, which gives that the second
term in the right side of \eqref{extLot} is nonnegative.  And finally,
the last term of \eqref{extLot} is easily seen to be non-negative on the
support of $(1-\beta(|x|/R_1))$.  We now account for the error term
that results when derivatives land on the cutoff.  Here we see that
\begin{multline}
  \label{extLotErr}
 \Bigl| -\frac{1}{2}a(x)^{-2}\partial_x[a(x)^2 \partial_x g(x)] + \frac{1}{4}(1-\beta(|x|/R_1))
\Bigl\{a(x)^{-2}\partial_x\Bigl[a(x)^2\partial_x\Bigl(h(x)a(x)^{-2}\partial_x
(a(x)^2)\Bigr)\Bigr]\Bigr\} \Bigr|\\\lesssim \rho^{-1} R_1^{-2} \mathbf{1}_{|x|\approx R_1}.
\end{multline}

Using each of these analyses in \eqref{ibp}, we have established
\begin{multline*}
- \int_0^T\int \Box_\g u \Bigl\{f(x)\partial_x u + g(x)
u\Bigr\}\,dV\,dt - \int \partial_t u \Bigl(f(x)\partial_x u + g(x) u\Bigr)\,dV\Bigl|_0^T
\\\gtrsim \rho^{-1}\int_0^T\int_{|x|\approx \rho} |\partial
u|^2\,dV\,dt + \rho^{-3} \int_0^T \int_{|x|\approx \rho} u^2\,dV\,dt -
\rho^{-1} R_1^{-2} \int_0^T\int_{|x|\approx R_1} u^2\,dV\,dt.
\end{multline*}
By the Schwarz inequality, we may bound, independently of $\rho$,
\[\int f(x) \partial_t u \partial_x u\,dV\lesssim E[u](t),\]
to which we can, in turn, apply \eqref{energy}.
The term
\[\frac{1}{2} \int a(x)^{-2}h(x)\partial_x ((1-\beta(|x|/R_1)) a(x)^2) u\partial_t
u\,dV\]
is handled similarly when combined with the following variant of a
Hardy inequality:
\begin{equation}\label{Hardy}\int a(x)^{-2} u^2\,dV\lesssim \int
  (\partial_x u)^2\,dV.
\end{equation}
To apply such, notice that 
\[a(x)^{-2}h(x)\partial_x (\beta(|x|) a(x)^2)\lesssim a(x)^{-1}.\]
To prove
\eqref{Hardy}, we integrate by parts to see
\[\int a(x)^{-2}u^2\,dV = -2\int x u \partial_x u \,dx\,d\sigma
\lesssim \int \Bigl|\frac{x}{a(x)}\Bigr| \Bigl|\frac{u}{a(x)}\Bigr| |\partial_x
u|\,dV\lesssim \Bigl(\int a(x)^{-2}u^2\,dV\Bigr)^{1/2} \Bigl(\int
(\partial_x u)^2\,dV\Bigr)^{1/2}.\]

Applying these bounds to the time boundary terms and combining with
\eqref{energy}, we have
\begin{multline*}
\sup_{t\in [0,T]} E[u](t) + \rho^{-1}\int_0^T\int_{|x|\approx \rho} |\partial
u|^2\,dV\,dt + \rho^{-3} \int_0^T \int_{|x|\approx \rho} u^2\,dV\,dt
\\\lesssim E[u](0) + \int_0^T\int |\Box_\g u|\Bigl(|\partial u| +
\frac{|u|}{a(x)}\Bigr)\,dV\,dt + \rho^{-1} R_1^{-1} \|u\|^2_{LE_{R_1}}.
\end{multline*}
Taking the supremum over $\rho\ge R$ yields \eqref{midext}.
\end{proof}

\subsection{High frequency estimate (Proposition \ref{HFprop})}

In this section, we prove a local energy estimate for any range of
time-frequencies that is bounded away from $0$. 
To begin, we establish the estimate on a large ball, which
will be supplemented with \eqref{midext}.

\begin{lem}\label{hf_lemma}
For any $R>0$ sufficiently large, we have
\begin{multline}
  \label{int}
\frac{1}{R}
\Bigl\|(1+x^{2m})^{-1/2} \partial_x u\Bigr\|^2_{L^2 L^2_{|x|<R/2}}
+\frac{1}{R^{4m+1}} \Bigl\||x|^m (1+x^{2m})^{-1/2} a(x)^{-1} |\ang_0
u|\Bigr\|^2_{L^2 L^2_{|x|<R/2}}
\\+\frac{1}{R}
 \Bigl\| |x|^m (1+x^{2m})^{-1/2} \partial_t
 u\Bigr\|^2_{L^2L^2_{|x|<R/2}}
+\frac{1}{R^{4m+1}} \Bigl\|\la x\ra^{-1} u\Bigr\|^2_{L^2
  L^2_{|x|<R/2}}
\\ \lesssim  E[u](0)+ \int_0^T\int_{|x|<R} |\Box_\g u| \Bigl(|\partial u| +
    |\la x\ra^{-1} u|\Bigr)\,dV\,dt + \|u\|_{LE^1_{|x|\approx R}}^2.
\end{multline}
\end{lem}

\begin{proof}
We apply \eqref{ibp} with all of the $w=\beta(|x|/R) u$, $f(x) =
\frac{x}{Ra(x/R^2)}$, and $g(x) = \frac{1}{2} f'(x) +
\frac{\delta}{R^{4m}} \frac{a'(x)}{a(x)}f(x)$, and we shall now
examine the coefficients of the last four terms in the right side of
\eqref{ibp}.  Here $\delta>0$ is a small parameter that will be fixed later.

For the coefficient of $(\partial_x w)^2$, noting that
$\frac{a'(x)}{a(x)}f(x)\ge 0$, we have
\[  f'(x)+g(x) - \frac{1}{2}\Bigl\{a(x)^{-2}\partial_x (a(x)^2
  f(x))\Bigr\} 
\ge f'(x) - \frac{a'(x)}{a(x)}f(x)
=\frac{1}{R}\frac{1-\Bigl(\frac{x}{R}\Bigr)^{4m}}{(1+x^{2m})\Bigl(1+\frac{x^{2m}}{R^{4m}}\Bigr)^{1+\frac{1}{2m}}}.
\]
In particular, this is nonnegative on the support of $\beta(|x|/R)$,
and
\begin{equation}\label{dx}\int_0^T\int \Bigl(f'(x) + g(x) - \frac{1}{2}\Bigl\{a(x)^{-2}\partial_x(a(x)^2
f(x))\Bigr\}\Bigr) (\partial_x w)^2 \,dV\,dt \gtrsim \frac{1}{R}
\Bigl\|(1+x^{2m})^{-1/2} \partial_x u\Bigr\|^2_{L^2 L^2_{|x|<R/2}}.
\end{equation}
Here and throughout this proof, all implicit constants are independent of $R$.

For the angular derivatives, we have
\[f(x)\frac{a'(x)}{a(x)} + g(x) - \frac{1}{2}\Bigl\{a(x)^{-2}\partial_x(a(x)^2
f(x))\Bigr\}= \frac{\delta}{R^{4m}}\frac{a'(x)}{a(x)} f(x)
=\frac{\delta}{R^{4m+1}} \frac{x^{2m}}{1+x^{2m}} \Bigl(1+\Bigl(\frac{x}{R^2}\Bigr)^{2m}\Bigr)^{-1/2m}.\]
And hence,
\begin{multline}\label{dang}\int_0^T\int \Bigl(f(x)\frac{a'(x)}{a(x)} + g(x) - \frac{1}{2}\Bigl\{a(x)^{-2}\partial_x(a(x)^2
f(x))\Bigr\}\Bigr) \frac{1}{a(x)^2}|\ang_0 w|^2\,dV\,dt\\\gtrsim
\frac{\delta}{R^{4m+1}} \Bigl\||x|^m (1+x^{2m})^{-1/2} a(x)^{-1} |\ang_0
u|\Bigr\|^2_{L^2 L^2_{|x|<R/2}}.\end{multline}
And for the time derivatives, it follows that
\[-g(x) + \frac{1}{2}\Bigl\{a(x)^{-2}\partial_x(a(x)^2
f(x))\Bigr\} =\Bigl(1-\frac{\delta}{R^{4m}}\Bigr)\frac{a'(x)}{a(x)}f(x)\]
and
\begin{equation}\label{dt}\int_0^T\int \Bigl(-g(x) + \frac{1}{2}\Bigl\{a(x)^{-2}\partial_x(a(x)^2
f(x))\Bigr\}\Bigr) (\partial_t w)^2\,dV\,dt \gtrsim \frac{1}{R}
 \Bigl\| |x|^m (1+x^{2m})^{-1/2} \partial_t
 u\Bigr\|^2_{L^2L^2_{|x|<R/2}}. 
\end{equation}

It remains to examine the lower order term, whose coefficient is
$-\frac{1}{2} a(x)^{-2}\partial_x(a(x)^2\partial_x g)$.  We first
compute
\begin{multline*}-\frac{1}{4} a(x)^{-2}\partial_x (a(x)^2\partial_x f'(x))
\\= \frac{(2m+1)\Bigl(\frac{x}{R^2}\Bigr)^{2m}}{4R x^2
  (1+x^{2m})\Bigl(1+\Bigl(\frac{x}{R^2}\Bigr)^{2m}\Bigr)^{3+\frac{1}{2m}}}
  \Bigl[x^{2m} - 1 - 2\Bigl(\frac{x}{R^2}\Bigr)^{2m} + 2m(1+x^{2m})\Bigl(1-\Bigl(\frac{x}{R^2}\Bigr)^{2m}\Bigr)\Bigr],
\end{multline*}
which can be observed to be nonnegative on $|x|<R$ for any $R$
sufficiently large.  Moreover
\[  -\frac{1}{4}\int_0^T\int \Bigl(a(x)^{-2}\partial_x[a(x)^2\partial_x
  f'(x)]\Bigr)w^2\,dV\,dt \gtrsim \frac{1}{R^{4m+1}} \Bigl\||x|^{m-1}
  (1+x^{2m})^{-1/2} \beta(|x|/R) u\Bigr\|^2_{L^2 L^2}.
\]
On the other hand,
\begin{multline*}-\frac{1}{2R^{4m}}a(x)^{-2}\partial_x \Bigl[a(x)^2\partial_x
\Bigl(\frac{a'(x)}{a(x)}f(x)\Bigr)\Bigr] 
= \frac{1}{2}\frac{x^{2m-2}}{R^{4m+1}(1+x^{2m})}
 \frac{1}{\Bigl(1+\Bigl(\frac{x}{R^2}\Bigr)^{2m}\Bigr)^{2+\frac{1}{2m}}}
\times \\  \Bigl\{ 4m^2 \frac{x^{2m}-1}{(x^{2m}+1)^2}
   \Bigl(1+\Bigl(\frac{x}{R^2}\Bigr)^{2m}\Bigr)^2 -
   \Bigl(\frac{x}{R^2}\Bigr)^{2m}\frac{1}{x^{2m}+1} \Bigl(1-x^{2m} +
   2\Bigl(\frac{x}{R^2}\Bigr)^{2m}\Bigr) \\+ \frac{2m}{(1+x^{2m})^2}\Bigl[1+
   5\Bigl(\frac{x}{R^2}\Bigr)^{2m} +
   x^{4m}\Bigl(\frac{x}{R^2}\Bigr)^{2m} +
   3\Bigl(\frac{x}{R^2}\Bigr)^{4m} 
+ x^{2m}\Bigl(-1+2\Bigl(\frac{x}{R^2}\Bigr)^{2m}+\Bigl(\frac{x}{R^2}\Bigr)^{4m}\Bigr)\Bigr]\Bigr\}.
\end{multline*}
On the support of $\beta(|x|/R)$ for $R$ sufficiently large, it follows that
\[\Bigl|-\frac{\delta}{2R^{4m}}a(x)^{-2}\partial_x \Bigl[a(x)^2\partial_x
\Bigl(\frac{a'(x)}{a(x)}f(x)\Bigr)\Bigr] \Bigr|\lesssim \delta
\frac{|x|^{2m-2}}{R^{4m+1} (1+x^{2m})} \quad\text{ for $|x|<R$}.\]
Thus, if $\delta>0$ is chosen sufficiently small, it follows that
\begin{equation}\label{lot}
  -\frac{1}{2}\int_0^T\int \Bigl(a(x)^{-2}\partial_x[a(x)^2\partial_x
  g]\Bigr)w^2\,dV\,dt \gtrsim \frac{1}{R^{4m+1}} \Bigl\||x|^{m-1}
  (1+x^{2m})^{-1/2} u\Bigr\|^2_{L^2 L^2_{|x|<R/2}}.
\end{equation}

We next provide a simple argument to show that the vanishing of the
coefficient at the origin in \eqref{lot} can be removed.  Indeed,
integration by parts gives
\[\int \beta(|x|) u^2 a(x)^2\,dx = -\int \beta'(|x|) |x| u^2 a^2(x)\,dx
- 2 \int \beta(|x|) x u \partial_x u a^2(x)\,dx - 2\int \beta(|x|)
\frac{x a'(x)}{a(x)} u^2 \, a(x)^2\,dx.\]
Applying the Schwarz inequality to the second term and bootstrapping
yields
\begin{equation}\label{lot2}\int \beta(|x|) u^2\,dV \lesssim \int |\beta'(|x|)| |x| u^2\,dV +
\int \beta(|x|) x^2 (\partial_x u)^2\,dV + \int \beta(|x|)
\frac{x^{2m}}{(1+x^{2m})} u^2\,dV.\end{equation}
Upon inclusion of a factor of $R^{-4m-1}$, each of the terms on the
right side can be controlled by the right sides of \eqref{dx} and \eqref{lot}.

Using \eqref{dx}, \eqref{dt}, \eqref{dang}, \eqref{lot}, and
\eqref{lot2} in \eqref{ibp} gives
\begin{multline}\label{ibp2}
    -\int_0^T\int \Box_\g w \Bigl\{f(x) \partial_x w +
g(x) w\Bigr\} \,dV\,dt - \int \partial_t w
    \Bigl(f(x)\partial_x w +
    g(x) w\Bigr)
    dV\Bigl|_0^T\\\gtrsim \frac{1}{R}
\Bigl\|(1+x^{2m})^{-1/2} \partial_x u\Bigr\|^2_{L^2 L^2_{|x|<R/2}}
+\frac{\delta}{R^{4m+1}} \Bigl\||x|^m (1+x^{2m})^{-1/2} a(x)^{-1} |\ang_0
u|\Bigr\|^2_{L^2 L^2_{|x|<R/2}}
\\+\frac{1}{R}
 \Bigl\| |x|^m (1+x^{2m})^{-1/2} \partial_t
 u\Bigr\|^2_{L^2L^2_{|x|<R/2}}
+\frac{1}{R^{4m+1}} \Bigl\|\la x\ra^{-1} \beta(|x|/R) u\Bigr\|^2_{L^2 L^2}.
  \end{multline}

Noting that $f(x)$ is bounded (independent of $R$) and $|g(x)|\lesssim
1/R$ on the support of $\beta(|x|/R)$, we can apply the Schwarz
inequality and \eqref{Hardy} to bound each of the time-boundary terms
by the energy at that time.  Thus, \eqref{energy} gives
\begin{multline}\label{interior}
   \frac{1}{R}
\Bigl\|(1+x^{2m})^{-1/2} \partial_x u\Bigr\|^2_{L^2 L^2_{|x|<R/2}}
+\frac{\delta}{R^{4m+1}} \Bigl\||x|^m (1+x^{2m})^{-1/2} a(x)^{-1} |\ang_0
u|\Bigr\|^2_{L^2 L^2_{|x|<R/2}}
\\+\frac{1}{R}
 \Bigl\| |x|^m (1+x^{2m})^{-1/2} \partial_t
 u\Bigr\|^2_{L^2L^2_{|x|<R/2}}
+\frac{1}{R^{4m+1}} \Bigl\|\la x\ra^{-1} \beta(|x|/R) u\Bigr\|^2_{L^2
  L^2}
\\\lesssim E[u](0) + \int_0^T\int |\Box_\g w| \Bigl(|\partial_t w| + |\partial_x w| +
    |\la x\ra^{-1} w|\Bigr)\,dV\,dt.
  \end{multline}

We finally examine the nonhomogeneous term and notice
\begin{equation}\label{inhom}\int_0^T \int |[\Box_\g, \beta(|x|/R)] u| \Bigl(|\partial w| + |\la
x\ra^{-1} w|\Bigr)\, dV\,dt 
\lesssim \|u\|^2_{LE^1_R},
\end{equation}
which completes the proof.\end{proof}

We may now establish our main high frequency estimate.

\begin{proof}[Proof of Proposition \ref{HFprop}]
  We apply \eqref{int} and \eqref{midext} (with $R_1=R/2$) to see that
  \begin{multline}
    \label{HFcombined}
    \frac{1}{R}
\Bigl\|(1+x^{2m})^{-1/2} \partial_x u_{>\tau}\Bigr\|^2_{L^2 L^2_{|x|<R}}
+\frac{\delta}{R^{4m+1}} \Bigl\||x|^m (1+x^{2m})^{-1/2} a(x)^{-1} |\ang_0
u_{>\tau} |\Bigr\|^2_{L^2 L^2_{|x|<R}}
\\+\frac{1}{R}
 \Bigl\| |x|^m (1+x^{2m})^{-1/2} \partial_t
 u_{>\tau} \Bigr\|^2_{L^2L^2_{|x|<R}}
+\frac{1}{R^{4m+1}} \Bigl\|\la x\ra^{-1} \beta(|x|/2R) u_{>\tau}\Bigr\|^2_{L^2
  L^2} 
+ \|u_{>\tau}\|^2_{LE^1_{|x|>R}}
\\ \lesssim  E[u](0)+ \int_0^T\int |\Box_\g u| \Bigl(|\partial u| +
    |\la x\ra^{-1} u|\Bigr)\,dV\,dt + R^{-2} \|u_{>\tau}\|_{LE_{|x|\approx R}}^2.
  \end{multline}
Since 
\[R^{-2} \|u_{>\tau}\|_{LE_{|x|\approx R}}^2 \lesssim \tau^{-2} R^{-3}
\|\partial_t u_{>\tau}\|^2_{L^2L^2_{|x|\approx R}},\]
we see that the error term may be bootstrapped provided that $R$ is
chosen sufficiently large (depending on $\tau$).  And from what results, the
desired estimate follows immediately.
\end{proof}

\subsection{Low frequency estimate (Proposition \ref{LFprop})}
Here we establish a local energy estimate for sufficiently low
time-frequencies.
We will again use \eqref{midext}, but we modify the
interior estimate that we couple with it.

\begin{lem}\label{lf_interior}
For any $R>0$, we have
\begin{multline}
  \label{lf}
  \|\la x\ra^{-1/2} \partial_x u\|^2_{L^2 L^2_{|x|<{R/2}}} + \|\la x\ra^{-1/2}
  a(x)^{-1}\ang_0 u\|^2_{L^2L^2_{|x|<{R/2}}} + \|\la x\ra^{-1/2} \partial_t
  u\|^2_{L^2 L^2_{|x|<{R/2}}} + \|\la x\ra^{-m-\frac{3}{2}}
  u\|^2_{L^2L^2_{|x|<{R/2}}} \\\lesssim E[u](0) + \int_0^T\int |\Box_\g
  u|\Bigl(|\partial_t u|+ \frac{1}{a(x)}|u|\Bigr)\,dV\,dt + \|\la
  x\ra^{-1/2} \partial_t u\|^2_{L^2 L^2_{|x| <R}} + \|
  u\|^2_{LE^1_{|x|\approx R}}.
\end{multline}
\end{lem}

\begin{proof}
Here we apply \eqref{ibp} with $f\equiv 0$ and $g(x) = 1/a(x)$.  This
gives
\begin{multline}\label{LagrangianCorrection}
  \int_0^T\int \frac{1}{a(x)} \Bigl[(\partial_x w)^2 +
  \frac{1}{a(x)^2}|\ang_0 w|^2\Bigr]\,dV\,dt + \frac{2m-1}{2}\int_0^T\int
  \frac{x^{2m-2}}{(1+x^{2m})^2}\frac{1}{a(x)} w^2\, dV\,dt
\\= -\int \frac{1}{a(x)}w\partial_t w\,dV|_0^T - \int_0^T \int
\frac{1}{a(x)} w \Box_\g w\, dV\,dt + \int_0^T \int
\frac{1}{a(x)} (\partial_t w)^2 \,dV\,dt.
\end{multline}
We may again apply \eqref{lot2} to eliminate the vanishing in the
coefficient of the third term in the left side.  Also applying the
Schwarz inequality, \eqref{Hardy}, and \eqref{energy}, we obtain
\begin{multline*}
  \int_0^T\int \frac{1}{a(x)} \Bigl[(\partial_x w)^2 +
  \frac{1}{a(x)^2}|\ang_0 w|^2\Bigr]\,dV\,dt + \int_0^T\int
\la x\ra^{-2m-2}  \frac{1}{a(x)} w^2\, dV\,dt
\\\lesssim E[w](0) + \int_0^T \int |\Box_\g w|\Bigl(|\partial_t w|+ \frac{1}{a(x)} |w|\Bigr) \, dV\,dt + \int_0^T\int
\frac{1}{a(x)} (\partial_t w)^2 \,dV\,dt.
\end{multline*}
The proof concludes by letting $w = \beta(|x|/R) u$ and using \eqref{inhom}.
\end{proof}

We now combine the preceding, \eqref{lf}, with the exterior estimate
\eqref{midext} in order to establish a local energy estimate for
sufficiently small frequencies.

\begin{proof}[Proof of Proposition \ref{LFprop}]
  We supplement \eqref{lf} (applied to $u_{<\tau}$) with
  \eqref{energy} and
  \eqref{midext} with $R_1=2$.  This gives
  \begin{multline*}
 \|\partial u_{<\tau}\|^2_{L^\infty L^2}+    \|\la x\ra^{-1/2} \partial u_{<\tau}\|^2_{L^2 L^2_{|x|<R}} + \|\la
    x\ra^{-m-\frac{3}{2}} u_{<\tau}\|_{L^2L^2_{|x|<R}}^2 +
    \|u_{<\tau}\|^2_{LE^1_{|x|>R}}\\\lesssim E[u](0) + \|\Box_\g
    u\|_{L^1L^2+LE^*} \|(\partial u_{<\tau}, a(x)^{-1}
    u_{<\tau})\|_{L^\infty L^2\cap LE}
\\+ \|\la x\ra^{-1/2} \partial_t u_{<\tau}\|^2_{L^2 L^2_{|x|<2R}} + R^{-1}
\|u_{<\tau}\|^2_{LE_{|x|\approx 1}}.
  \end{multline*}
Here we have also applied the Schwarz inequality to the inhomogeneous term.
By choosing $R$ sufficiently large, we may bootstrap the last term in
the right side.  The remaining error term can then be absorbed
provided that $\tau$ is sufficiently small (depending on $R$).
\end{proof}

\section{An improved estimate - Proof of Theorem \ref{thm_sharp}}

In order to prove Theorem \ref{thm_sharp}, we define $P=
-\partial_t^2 + \partial_x^2 + a(x)^{-2}\ang_0\cdot\ang_0$ and shall prove:
\begin{prop}\label{sharp_trapping_prop}
  Suppose that $v(t,\cd)$ is supported in $\{|x|<1\}$, that
  $v(t,x)\equiv 0$ for $t\le 0$, and $P v = g$ where $g(t,\cd)$
  is supported in $\{|x|<1\}$ and $g(t,x)\equiv 0$ for $t\le 0$.
  Then, for $\delta>0$,
  \begin{equation}
    \label{sharp}
   \|\partial_x v\|_{LE} +  \|v\|_{\la D_\omega\ra^{\frac{m-1}{2(m+1)}} LE^1} \lesssim
    \|g\|_{\la D_\omega\ra^{-\frac{m-1}{2(m+1)}} LE^* + (|x|/\la
      x\ra)^{(m-1+\delta)/2}LE^*}.
  \end{equation}
\end{prop}

We will first show that this suffices to prove Theorem
\ref{thm_sharp}.  To begin, we shall argue that if $v(t,\cd)$ is
supported in $\{|x|<\varepsilon\}$ with $\varepsilon>0$ sufficiently
small then we may replace $P$ by $\Box_\g$ in the above proposition.
Indeed, if $\Box_\g v = g$, then $Pv = g -
2\frac{a'(x)}{a(x)}\partial_x v$.  Then \eqref{sharp} gives
\[\|\partial_x v\|_{LE} + \|v\|_{\la
  D_\omega\ra^{\frac{m-1}{2(m+1)}}LE^1} \lesssim \|g\|_{\la
  D_\omega\ra^{-\frac{m-1}{2(m+1)}} LE^* + (|x|/\la
  x\ra)^{\frac{m-1+\delta}{2}}LE^*} +
\||x|^{\frac{3m-1-\delta}{2}} \partial_x v\|_{LE^*}.\]
Due to the restricted support of $v$, this last term is
$\varepsilon^{\frac{3m-1-\delta}{2}} \O(\|\partial_x v\|_{LE})$ and
may be bootstrapped provided $\varepsilon$ is small enough.

We next argue that the $L^1L^2$ norm may be
added in the right side.  This follows trivially from
\eqref{main_lossy} for the $\partial_x$ components.  
Suppose that $\phi$ solves $\Box_\g \phi = f$
backwards in time and that $\phi(t,x), f(t,x)\equiv 0$ for $t\ge T$.
Then by \eqref{main_lossy2} and \eqref{sharp}, it follows that
\[\|\partial \phi\|_{L^\infty L^2} + \|(|x|/\la x\ra)^m \partial\phi\|_{LE} \lesssim \|\la
D_\omega\ra^{\frac{m-1}{2(m+1)}} f\|_{LE^*}.\]
Hence,
\[\la \partial_t v, f\ra = - \la \Box_\g v, \partial_t\phi\ra = - \la
g, \partial_t\phi\ra 
\le \|g\|_{L^1L^2} \|\partial_t \phi\|_{L^\infty L^2}
\lesssim \|g\|_{L^1
  L^2} \|\la D_\omega\ra^{\frac{m-1}{2(m+1)}} f\|_{LE^*},
\]
and thus by duality,
\begin{equation}\label{dtbound}\| \la D_\omega\ra^{-\frac{m-1}{2(m+1)}} \partial_t v\|_{LE} \lesssim \|g\|_{L^1L^2}.
\end{equation}

Applying \eqref{LagrangianCorrection} with $w=\la
D_\omega\ra^{-\frac{m-1}{2(m+1)}} v$, \eqref{energy}, and
\eqref{Hardy}, we have
\begin{equation}\label{domegabound}\|\la D_\omega\ra^{-\frac{m-1}{2(m+1)}} |\ang_0 v|\|^2_{LE} \lesssim
\|\Box_\g v\|_{L^1L^2} \|\la D_\omega\ra^{-\frac{m-1}{m+1}}
v\|_{L^\infty L^2} + \|\la D_\omega\ra^{-\frac{m-1}{2(m+1)}}\partial_t
v\|^2_{LE}.
\end{equation}
As \eqref{Hardy} and \eqref{energy} permit us to bound
\[\|\la D_\omega\ra^{-\frac{m-1}{m+1}}v\|_{L^\infty L^2} \lesssim
\|a(x)^{-1} u\|_{L^\infty L^2}\lesssim \|g\|_{L^1L^2},\]
\eqref{dtbound} and \eqref{domegabound} give
\begin{equation}\label{angbound}\| \la D_\omega\ra^{-\frac{m-1}{2(m+1)}} |\ang_0 v|\|_{LE} \lesssim \|g\|_{L^1L^2}.
\end{equation}

We have therefore sharpened \eqref{sharp} to say that if $\Box_\g
v=g_1+g_2$ where $g_1\in (|x|/\la x\ra)^{\frac{m-1+\delta}{2}} LE^*$,
$g_2\in L^1L^2$ and if $v(t,x), g_1(t,x), g_2(t,x)$ vanish identically
for $t\le 0$ and if each is supported where
$|x|<\varepsilon$, then
\begin{equation}
  \label{sharp2}
  \|\partial_x v\|_{LE} + \|\la D_\omega\ra^{-\frac{m-1}{2(m+1)}} \partial v\|_{LE} \lesssim
  \inf_{g=g_1+g_2} \Bigl(\|(\la x\ra/|x|)^{\frac{m-1+\delta}{2}} g_1\|_{LE^*}
  + \|g_2\|_{L^1L^2}\Bigr).
\end{equation}

We now proceed to make a few
reductions in order to use \eqref{sharp2} to prove Theorem
\ref{thm_sharp}.  First, it suffices to replace $u$ by $v =
(1-\beta(t)) \beta(|x|/\varepsilon) u$ as
$\|(1-\beta(|x|/\varepsilon))u\|_{LE^1}$ is trivially controlled by the
left side of \eqref{main_lossy}, which provides a better bound, and,
using a Hardy inequality,
$\|\beta(t)\beta(|x|/\varepsilon) u\|^2_{LE^1} \lesssim \sup_{0\le t\le 1}
E[u](t)$, which is appropriately bounded using \eqref{energy}.
We now observe that $v$ satisfies
\begin{multline*}\Box_\g v= 
2 \varepsilon^{-1}\text{ sgn}(x)
\beta'(|x|/\varepsilon)(1-\beta(t)) \partial_x u
+2 \varepsilon^{-1}\text{sgn}(x) (a'(x)/a(x)) \beta'(|x|/\varepsilon)
(1-\beta(t)) u 
\\+ \varepsilon^{-2} \beta''(|x|/\varepsilon)(1-\beta(t)) u + 2 \beta(|x|/\varepsilon)\beta'(t) \partial_t u
+ \beta(|x|/\varepsilon)\beta''(t) u.
\end{multline*}
Then \eqref{sharp2} gives that 
\[  \|\la D_\omega\ra^{-\frac{m-1}{2(m+1)}} \partial
    (1-\beta(t))\beta(|x|) u\|_{LE} \lesssim_\varepsilon \|\partial_x u\|_{LE} +
    \|\la x\ra^{-1} u\|_{LE}  + \sup_{1/2\le t\le 1} E[u](t)^{1/2}
\]
and subsequent applications of \eqref{energy} and \eqref{main_lossy}
complete the proof of Theorem \ref{thm_sharp}.

\subsection{WKB Analysis}
It now suffices to establish Proposition \ref{sharp_trapping_prop}.  The
method that we shall employ is inspired by \cite{MMTT_Str_BH}.  We
shall develop energy functionals that are based in WKB analysis in
order to achieve \eqref{sharp}.

We begin with $P v = g$.  Upon taking the Fourier transform in $t$ and
expanding into spherical harmonics, it suffices to examine
\begin{equation}\label{equation_tau_lambda} \partial_x^2\phi_{\lambda,\tau} +
V_{\lambda,\tau}(x)\phi_{\lambda,\tau} = g_{\lambda, \tau}
\quad\text{where } \quad
V_{\lambda,\tau} = \tau^2 - \frac{\lambda^2}{a(x)^2}.
\end{equation}
Here $\tau$ is the Fourier dual variable to $t$, and $\lambda^2$ runs
over the eigenvalues of $-\angdelta_{\S^2}$.
And to prove \eqref{sharp}, if $g_{\lambda,\tau} = g_{1,\lambda,\tau}
+ g_{2,\lambda,\tau}$ with $g_{1,\lambda,\tau}\in L^2_x$ and
$g_{2,\lambda,\tau}\in (|x|/\la x\ra)^{\frac{m-1+\delta}{2}}L^2_x$, it will suffice to show
\begin{equation}
  \label{sharp_tau_lambda}
  \la \lambda\ra^{-\frac{m-1}{2(m+1)}}
  (|\tau|+\lambda)\|\phi_{\lambda,\tau}\|_{L^2_x} + \|\partial_x \phi_{\lambda,\tau}\|_{L^2_x}
  \lesssim \la \lambda\ra^{\frac{m-1}{2(m+1)}} \|g_{1,\lambda,\tau}\|_{L^2_x}
  + \Bigl\|\Bigl(\frac{\la x\ra}{|x|}\Bigr)^{\frac{m-1+\delta}{2}} g_{2,\lambda,\tau}\Bigr\|_{L^2_x}.
\end{equation}
For ease of notation, we shall drop the subscripts in the sequel.
The analysis will be broken into four cases depending on the
relationship between $\tau$ and $\lambda$.

In the first three cases, we will prove the stronger estimate
\begin{equation}
  \label{wtsI-III}
  \|\partial_x \phi\|_{L^2} + (|\tau|+\lambda) \|\phi\|_{L^2} \lesssim \|g\|_{L^2}.
\end{equation}

\noindent{{\bf Case I:} ($\lambda, \tau \lesssim 1$)  Define the
  positive definite energy functional $E[\phi](x) = (\partial_x \phi)^2 + \phi^2$.
Then, since $|V_{\lambda,\tau}|\lesssim 1$ here,
\[\partial_x E[\phi](x) = 2\partial_x \phi (\partial_x^2\phi + \phi) \lesssim
g^2 + E[\phi](x).\]
By Gronwall's inequality and the fact that $\phi$ is compactly supported, it follows that
\[(\partial_x \phi)^2 + \phi^2 \lesssim \|g\|^2_{L^2},\]
from which the desired result is immediate upon integrating.

\noindent{\bf Case II:} ($\lambda \ll \tau$)  Here we similarly define
$E[\phi](x) = (\partial_x \phi)^2 + V_{\lambda,\tau}(x)\phi^2$.  Then,
\[\partial_x E[\phi]\lesssim |g\partial_x \phi| + \phi^2 |\partial_x
V_{\lambda,\tau}(x)|.\]
Since $|\partial_x V_{\lambda, \tau}(x)|\lesssim V_{\lambda,\tau}(x)$,
this gives
\[\partial_x E[\phi](x) \lesssim E[\phi](x) + g^2.\]
The desired estimate then follows, as above, from Gronwall's inequality and
integrating both sides of what results.

\noindent{\bf Case III:} ($\tau\ll \lambda$)  Here, due to the
different sign of $V_{\lambda,\tau}$ in this regime, we instead
consider the boundary value problem
\[
\begin{cases}
  \partial_x^2 \phi + V_{\lambda,\tau} \phi = g,\\
\phi(-1)=\phi(1)=0.
\end{cases}
\]
Multiplying both sides of the equation by $-\phi$ and integrating, we
find
\[\int_{-1}^{1} (\partial_x \phi)^2\, dx - \int_{-1}^{1}
V_{\lambda,\tau} \phi^2\,dx = -\int_{-1}^{1} g\phi\,dx.\]
In this case, $V_{\lambda,\tau} \approx -\lambda^2$ on the support of $\phi$, and the desired
result follows from the Schwarz inequality.

\noindent{\bf Case IV:} ($\tau \approx \lambda \gg 1$) 
Here we write 
\[V(x) = \lambda^2 \Bigl(1-\frac{1}{(x^{2m}+1)^{1/m}} +
  \varepsilon\Bigr) = \lambda^2 (b(x)+\varepsilon),\]
where $|\varepsilon|\lesssim 1$.  And we shall further subdivide this
analysis based on the relationship between $\lambda$ and
$\varepsilon$: $\varepsilon\ge -C\lambda^{-2m/(1+m)}$ and
 $\varepsilon\le -C\lambda^{-2m/(1+m)}$ where $C>0$ is a
sufficiently large constant.

{\em Case A:} ($\varepsilon\ge -C\lambda^{-2m/(1+m)}$)  
Here we shall show 
\begin{multline}
  \label{caseIVab}
  \lambda \|(\lambda^{-2m/(1+m)} +
  |b(x)+\varepsilon|)^{1/4}\phi\|_{L^\infty_x} +
  \|(\lambda^{-2m/(1+m)}+|b(x)+\varepsilon|)^{-1/4} \partial_x\phi\|_{L^\infty_x} \\\lesssim
  \|(\lambda^{-2m/(1+m)} + |b(x)+\varepsilon|)^{-1/4}g\|_{L^1_x}.
\end{multline}
And due to the following lemma, this will imply
\eqref{sharp_tau_lambda} in this case.

\begin{lem}
 For $m>1$ and $\varepsilon > -C \lambda^{-2m/(1+m)}$, where $C>0$, we have
\begin{equation}\label{bint}\|(\lambda^{-2m/(1+m)} + |b(x)+\varepsilon|)^{-1/4}
\|_{L^2(\{|x|\le 1\})} \lesssim \lambda^{(m-1)/2(m+1)}\end{equation}
and for any $\delta>0$
\[\|(\lambda^{-2m/(1+m)} + |b(x)+\varepsilon|)^{-1/4} |x|^{\frac{m-1+\delta}{2}}
\|_{L^2(\{|x|\le 1\})} =\O(1).\]
\end{lem}

\begin{proof}
We shall begin by examining the case that $b(x)\le
2C\lambda^{-2m/(1+m)}$.  From this, we have $|x|\lesssim
\lambda^{-1/(1+m)}$ and
\[\int_{\{x\,:\, b(x)\le
2C\lambda^{-2m/(1+m)}\}}
\frac{1}{(\lambda^{-2m/(1+m)}+|b(x)+\varepsilon|)^{-1/2}}\, dx
\lesssim \int_{|x|\lesssim \lambda^{-1/(1+m)}} \lambda^{m/(1+m)} \,dx
\approx \lambda^{(m-1)/(m+1)},
\]
which gives the first estimate.  Similarly,
\[\int_{\{x\,:\, b(x)\le
2C\lambda^{-2m/(1+m)}\}}
\frac{|x|^{m-1+\delta}}{(\lambda^{-2m/(1+m)}+|b(x)+\varepsilon|)^{-1/2}}\, dx
\lesssim \int_{|x|\lesssim \lambda^{-1/(1+m)}} \lambda^{\frac{1-\delta}{m+1}} \,dx
=\O(1).\]
  
We next consider when $b(x) \ge 2C\lambda^{-2m/(1+m)}$.  In this
setting, we have that $|b(x)+\varepsilon|\ge |b(x)|/2$ and that
$|x|\gtrsim \lambda^{-1/(1+m)}$.  Then we have
\[\int_{\{x\,:\, b(x)\ge
2C\lambda^{-2m/(1+m)}\}}\frac{1}{(\lambda^{-2m/(1+m)}+|b(x)+\varepsilon|)^{-1/2}}\, dx
\lesssim \int^1_{c\lambda^{-1/(1+m)}} b(x)^{-1/2}\,dx.\]
Note that
\[b(x)^{-1/2} = \frac{(x^{2m}+1)^{1/2m}}{((x^{2m}+1)^{1/m}-1)^{1/2}}.\]
Moreover, we have
\begin{equation}\label{tay}(x^{2m}+1)^{1/m}-1 = \int_0^{|x|} \frac{2
  y^{2m-1}}{(y^{2m}+1)^{1-\frac{1}{m}}}\,dy \gtrsim \int_0^{|x|}
y^{2m-1}\,dy \approx |x|^{2m}.
\end{equation}
We also note for later purposes that the quantity on the left is $\O(|x|^{2m})$ on the
support of $\phi$.
In particular, on $|x|\le 1$, $b(x)^{-1/2} \lesssim |x|^{-m}$.  This gives that
\[\|(\lambda^{-2m/(1+m)} + |b(x)+\varepsilon|)^{-1/4}
\|^2_{L^2(\{x\,:\, b(x)\ge
2C\lambda^{-2m/(1+m)}\})} \lesssim \int_{c\lambda^{-1/(1+m)}}^1
|x|^{-m}\,dx \lesssim \lambda^{\frac{m-1}{m+1}}\]
as desired for the first estimate, and 
 \[\|(\lambda^{-2m/(1+m)} + |b(x)+\varepsilon|)^{-1/4} |x|^{\frac{m-1+\delta}{2}}
\|^2_{L^2(\{x\,:\, b(x)\ge
2C\lambda^{-2m/(1+m)}\})} \lesssim \int_{c\lambda^{-1/(1+m)}}^1
|x|^{-1+\delta}\,dx =\O(1).\]
\end{proof}

We now proceed to proving \eqref{caseIVab}.
\begin{proof}[Proof of \eqref{caseIVab}]
Here we consider two subcases.

{\em Case i:} ($b(x)\ge 2C \lambda^{-2m/(1+m)}$)  We define
\[E[\phi] = \lambda^2 (b(x)+\varepsilon)^{1/2}\phi^2 +
(b(x)+\varepsilon)^{-1/2}(\partial_x \phi)^2 +
\frac{1}{2}b'(x)(b(x)+\varepsilon)^{-3/2}\phi\partial_x \phi.\]

Since $\varepsilon\ge -C \lambda^{-2m/(1+m)}$, it follows that $1\ll
\lambda^2 (b(x)+\varepsilon)^{(m+1)/m}$.  Since \eqref{tay} shows that
$|b'(x)|^2\lesssim b(x)^{(2m-1)/m}$, on $|x|\le 1$, we have
\[\frac{1}{2}(b'(x))^2 (b(x)+\varepsilon)^{-5/2} \le \lambda^2
(b(x)+\varepsilon)^{1/2}\]
provided $C$ is sufficiently large.
Thus, by observing that
\[\frac{1}{2} b'(x) (b(x)+\varepsilon)^{-3/2} \phi \partial_x \phi \ge
-\frac{1}{4} (b(x)+\varepsilon)^{-1/2} (\partial_x \phi)^2 -
\frac{1}{4}(b'(x))^2 (b(x)+\varepsilon)^{-5/2}\phi^2,\]
it follows that $E[\phi]$ is positive.

We can compute
\begin{multline}
  \label{dEdx}
  \frac{dE}{dx} = \Bigl(\frac{1}{2}b''(x) (b(x)+\varepsilon)^{-3/2} -
  \frac{3}{4} (b'(x))^2(b(x)+\varepsilon)^{-5/2}\Bigr)\phi \partial_x
  \phi
\\+ \Bigl(2(b(x)+\varepsilon)^{-1/2} \partial_x \phi +
\frac{1}{2}b'(x)(b(x)+\varepsilon)^{-3/2} \phi\Bigr) g.
\end{multline}
From this and that $b'(x)(b(x)+\varepsilon)^{-3/2}\lambda^{-1} =
\O(1)$ on $|x|\le 1$, which is a consequence of \eqref{tay}, it follows that
\begin{equation}\label{coerenergy}  \frac{dE}{dx} \lesssim \lambda^{-1} b''(x) (b(x)+\varepsilon)^{-3/2}
  E + \lambda^{-1} (b'(x))^2 (b(x)+\varepsilon)^{-5/2} E
+ (b(x)+\varepsilon)^{-1/4} g E^{1/2}.
\end{equation}
Using \eqref{tay}, we observe that
\[\lambda^{-1} b''(x) (b(x)+\varepsilon)^{-3/2}
  + \lambda^{-1} (b'(x))^2 (b(x)+\varepsilon)^{-5/2} \lesssim
  (b(x)+\varepsilon)^{-1/2 - 1/m} \lambda^{-1} \lesssim
  (b(x)+\varepsilon)^{-1/2} \lambda^{-(m-1)/(m+1)}.\]
So,
\[\frac{d\, E^{1/2}}{dx} \lesssim \lambda^{-(m-1)/(m+1)}
(b(x)+\varepsilon)^{-1/2} E^{1/2} + (b(x)+\varepsilon)^{-1/4} g.\]
Upon integrating (if $x>0$, we apply the negative of the integral from
$1$ to $x$, and if $x<0$, we integrate from $-1$ to $x$) 
and applying Gronwall's inequality, since
\[\lambda^{-(m-1)/(m+1)} \int (b(x)+\varepsilon)^{-1/2}\,dx = \O(1)\]
by \eqref{bint},
it follows that
\[E^{1/2} \lesssim \int (b(x)+\varepsilon)^{-1/4} g\,dx.\]
This estimate implies \eqref{caseIVab}.

{\em Case ii:} ($b(x)\le 2C\lambda^{-2m/(1+m)}$) Defining
\[E[\phi] = \lambda^{(m+2)/(m+1)}\phi^2 + \lambda^{m/(m+1)}
(\partial_x\phi)^2,\]
it follows that
\[\frac{dE}{dx} = 2 \lambda^{(m+2)/(m+1)} \phi \partial_x \phi +
2\lambda^{m/(m+1)} \partial_x \phi (g-\lambda^2
(b+\varepsilon)\phi)\lesssim \lambda^{1/(m+1)} E +
\lambda^{m/2(m+1)}|g| E^{1/2}.\]
Dividing through by $E^{1/2}$, integrating, and applying Gronwall's inequality, where
\eqref{bint} shows that the integral in the resulting exponential is
$\O(1)$, we have
\[E^{1/2} \lesssim \|\lambda^{m/2(m+1)} g\|_{L^1},\]
which gives \eqref{caseIVab} in this range.
\end{proof}

{\em Case B:} ($\varepsilon \le -C\lambda^{-2m/(1+m)}$)
In this remaining case, we will show 
\begin{multline}
  \label{caseIVc}
  \lambda \|(\lambda^{-2/3} |\varepsilon|^{(2m-1)/3m} +
  |b(x)+\varepsilon|)^{1/4}\phi\|_{L^\infty_x} 
+\|(\lambda^{-2/3} |\varepsilon|^{(2m-1)/3m} +
  |b(x)+\varepsilon|)^{-1/4} \partial_x \phi\|_{L^\infty_x} 
\\\lesssim
  \|(\lambda^{-2/3} |\varepsilon|^{(2m-1)/3m}  + |b(x)+\varepsilon|)^{-1/4}g\|_{L^1_x},
\end{multline}
which, as above, will imply \eqref{sharp_tau_lambda} when combined
with the lemma:

\begin{lem}
 For $m>1$ and $\varepsilon \le -C \lambda^{-2m/(1+m)}$, where $C>0$, we have
\begin{equation}\label{bintc}\|((\lambda^{-2/3} |\varepsilon|^{(2m-1)/3m} + |b(x)+\varepsilon|)^{-1/4}
\|_{L^2(\{|x|\le 1\})} \lesssim \lambda^{(m-1)/2(m+1)}\end{equation}
and for any $\delta>0$
\begin{equation}\label{bintcii}\|((\lambda^{-2/3} |\varepsilon|^{(2m-1)/3m} + |b(x)+\varepsilon|)^{-1/4} |x|^{\frac{m-1+\delta}{2}}
\|_{L^2(\{|x|\le 1\})} =\O(1).
\end{equation}
\end{lem}

\begin{proof}
Here, as we will below, we consider three separate cases.  For
convenience, we will abbreviate $\alpha = C\lambda^{-2/3}|\varepsilon|^{\frac{2m-1}{3m}}$.

{\em Case i:} ($b+\varepsilon \ge C \lambda^{-2/3}|\varepsilon|^{(2m-1)/3m}$)
Using that $\varepsilon<0$ here, we have $b'(x)\gtrsim
(b(x))^{\frac{2m-1}{2m}} \gtrsim
(b(x)+\varepsilon)^{\frac{2m-1}{2m}}$.  Thus,
\[
\int_{b+\varepsilon \ge \alpha}
\frac{dx}{(b(x)+\varepsilon)^{1/2}}
 \lesssim \int_{\substack{b+\varepsilon \ge  \alpha\\x>0}}
 \frac{b'(x)}{(b(x)+\varepsilon)^{\frac{1}{2}+\frac{2m-1}{2m}}}\,dx
\lesssim (\alpha)^{-\frac{m-1}{2m}}
\lesssim \lambda^{(m-1)/(m+1)},
\]
which establishes \eqref{bintc}.  For \eqref{bintcii}, we similarly
argue
\[\int_{b+\varepsilon \ge \alpha}
\frac{|x|^{m-1+\delta}}{(b(x)+\varepsilon)^{1/2}}\,dx
\lesssim \int_{\substack{b+\varepsilon \ge \alpha\\x>0}}
\frac{(b')^{\frac{m-1+\delta}{2m-1}}}{(b(x)+\varepsilon)^{1/2}}\,dx 
\lesssim \int_{\substack{b+\varepsilon\ge \alpha\\x>0}}
\frac{b'(x)}{(b(x)+\varepsilon)^{1-\frac{\delta}{2m}}}\,dx
\lesssim |\varepsilon|^{\delta/2m}.
\]

{\em Case ii:} ($|b+\varepsilon| \le C
\lambda^{-2/3}|\varepsilon|^{(2m-1)/3m}$)
We note that
  $|b+\varepsilon|\le \alpha$ if and only if
  \[\Bigl[\frac{1}{[1-(|\varepsilon|-\alpha)]^m}-1\Bigr]^{1/2m}\le
  |x|\le \Bigl[\frac{1}{[1-(|\varepsilon|+\alpha)]^m}-1\Bigr]^{1/2m}.\]
The length of this interval satisfies:
\begin{equation}\label{caseii}\begin{split}
\Bigl[\frac{1}{[1-(|\varepsilon|+\alpha)]^m}-1\Bigr]^{1/2m}-\Bigl[\frac{1}{[1-(|\varepsilon|-\alpha)]^m}
  -1\Bigr]^{1/2m}
&=
\frac{1}{2m}\int^{\frac{1}{[1-(|\varepsilon|+\alpha)]^m}-1}_{\frac{1}{[1-(|\varepsilon|-\alpha)]^m}-1}
x^{-1+\frac{1}{2m}}\,dx\\
&\lesssim |\varepsilon|^{-1+\frac{1}{2m}}
  \Bigl(\frac{1}{[1-(|\varepsilon|+\alpha)]^m} -
  \frac{1}{[1-(|\varepsilon|-\alpha)]^m}\Bigr)\\ &\lesssim
  |\varepsilon|^{-1+\frac{1}{2m}} \alpha \approx \lambda^{-\frac{2}{3}}|\varepsilon|^{-\frac{2m-1}{6m}}.
\end{split}
\end{equation}
Then
\[\int_{|b+\varepsilon|\le \alpha} \lambda^{1/3}
|\varepsilon|^{-(2m-1)/6m}\,dx
\lesssim \lambda^{-1/3} |\varepsilon|^{-(2m-1)/3m} \lesssim \lambda^{(m-1)/(m+1)},
\]
which establishes \eqref{bintc}.

Similarly, for \eqref{bintcii}, we have
\begin{align*}\int_{|b+\varepsilon|\le \alpha} \lambda^{1/3}
&|\varepsilon|^{-(2m-1)/6m} |x|^{m-1+\delta}\,dx \\
&\lesssim \lambda^{1/3}|\varepsilon|^{-(2m-1)/6m} \Bigl\{\Bigl[\frac{1}{[1-(|\varepsilon|+\alpha)]^m}-1\Bigr]^{(m+\delta)/2m}-\Bigl[\frac{1}{[1-(|\varepsilon|-\alpha)]^m}
  -1\Bigr]^{(m+\delta)/2m}\Bigr\}\\
&\lesssim \lambda^{1/3} |\varepsilon|^{-(5m-1)/6m}
  |\varepsilon|^{\delta/2m} \alpha\\ &\approx \lambda^{-1/3}
  |\varepsilon|^{-(m+1)/6m} |\varepsilon|^{\delta/2m} \lesssim |\varepsilon|^{\delta/2m}.
\end{align*}

{\em Case iii:} ($b+\varepsilon \le - C\lambda^{-2/3} |\varepsilon|^{(2m-1)/3m}$)
In this region, we have $|b+\varepsilon|\ge C\lambda^{-2/3}
|\varepsilon|^{\frac{2m-1}{3m}}$, $|\varepsilon| >
C\lambda^{-2m/1+m}$, and $|x|\lesssim |\varepsilon|^{1/2m}$.  Thus, since
$m\ge 2$,
\[  \int_{|b(x)+\varepsilon|\ge
  C\lambda^{-2/3}|\varepsilon|^{\frac{2m-1}{3m}}}
  \frac{dx}{|b(x)+\varepsilon|^{1/2}} \lesssim \lambda^{1/3}
  |\varepsilon|^{-\frac{2m-1}{6m}}|\varepsilon|^{1/2m}
\lesssim \lambda^{\frac{m-1}{m+1}},
\]
which gives \eqref{bintc}.  For \eqref{bintcii}, by factoring the
denominator, we observe
\begin{align*}
  \int_{b(x)+\varepsilon\le-\alpha}
  \frac{|x|^{m-1+\delta}}{|b(x)+\varepsilon|^{1/2}} \, dx&\lesssim
                                                      \frac{1}{|\varepsilon|^{1/4}}
                                                      \int_{b(x)+\varepsilon\le-\alpha}
                                                      \frac{|x|^{m-1+\delta}}{(|\varepsilon|^{1/2}-(b(x))^{1/2})^{1/2}}\,
                                                           dx\\
&\lesssim \frac{1}{|\varepsilon|^{1/4}}
  \int_{b(x)+\varepsilon\le-\alpha} \frac{b'(x)}{b(x)^{1/2}
  (|\varepsilon|^{1/2}-(b(x))^{1/2})^{1/2}}\, dx = \O(1).
\end{align*}
\end{proof}

\begin{proof}[Proof of \eqref{caseIVc}]
We will split the proof into three cases.  For the first case, we
argue in a fashion similar to that of Case i of \eqref{caseIVab}.
And in the second case, the argument mirrors that of Case ii of
\eqref{caseIVab}.  In the remaining case, the sign of the potential
term $b(x)+\varepsilon$ changes, and we use arguments common to
elliptic equations.

{\em Case i: } ($b+\varepsilon \ge 2C \lambda^{-2/3} |\varepsilon|^{(2m-1)/3m} $)  
Here, as in Case {\em A(i)}, we have $|b'(x)|^2\lesssim
b(x)^{(2m-1)/m}$.  Then we see that
\[|b'(x)|^2\lesssim |b(x)+\varepsilon|^{(2m-1)/m} +
|\varepsilon|^{(2m-1)/m}
\le |b(x)+\varepsilon|^3 \Bigl(|b(x)+\varepsilon|^{-(m+1)/m} +
|\varepsilon|^{(2m-1)/m} |b(x)+\varepsilon|^{-3}\Bigr).\]
Thus, in this case, we have
\begin{equation}\label{coercive}|b'(x)|^2\lesssim \lambda^2 |b(x)+\varepsilon|^3.\end{equation}
If we define,
\[E[\phi] = \lambda^2 (b(x)+\varepsilon)^{1/2} \phi^2 +
(b(x)+\varepsilon)^{-1/2}(\partial_x\phi)^2 + \frac{1}{2}b'(x)
(b(x)+\varepsilon)^{-3/2}\phi \partial_x \phi,\]
it follows from \eqref{coercive} that $E[\phi]$ is positive provided $C$ is
sufficiently large.

We then compute
\[
\frac{dE}{dx} \lesssim \lambda^{-1} b''(x) (b(x)+\varepsilon)^{-3/2}
  E + \lambda^{-1} (b'(x))^2 (b(x)+\varepsilon)^{-5/2} E
+ (b(x)+\varepsilon)^{-1/4} g E^{1/2},
\]
and an application of Gronwall's inequality will yield \eqref{caseIVc}
provided that we can show
\begin{equation}
  \label{caseIVci_goal}
  \int_{b(x)+\varepsilon \ge 2C\lambda^{-2/3}|\varepsilon|^{(2m-1)/3m}}
\Bigl(\lambda^{-1} b''(x) (b(x)+\varepsilon)^{-3/2}
   + \lambda^{-1} (b'(x))^2 (b(x)+\varepsilon)^{-5/2} \Bigr)\,dx = \O(1).
\end{equation}
By symmetry, it suffices to demonstrate this where $x>0$.

We first note that
\[(b'(x))^2 (b(x)+\varepsilon)^{-5/2} 
\lesssim b'(x) (b(x)+\varepsilon)^{-(3m+1)/2m} +
|\varepsilon|^{\frac{2m-1}{2m}}b'(x) (b(x)+\varepsilon)^{-5/2}.\]
Thus using the change of variables $u=b(x) + \varepsilon$, we can
evaluate
\begin{multline*}\lambda^{-1}\int_{b(x)+\varepsilon \ge
  2C\lambda^{-2/3}|\varepsilon|^{(2m-1)/3m}} (b'(x))^2
(b(x)+\varepsilon)^{-5/2} \,dx 
\\\lesssim \lambda^{-1}
(\lambda^{-2/3}|\varepsilon|^{\frac{2m-1}{3m}})^{-\frac{m+1}{2m}} +
\lambda^{-1} |\varepsilon|^{\frac{2m-1}{2m}} (\lambda^{-2/3}
|\varepsilon|^{\frac{2m-1}{3m}})^{-3/2} = \O(1).
\end{multline*}

For the other term in \eqref{caseIVci_goal}, we argue similarly.
Here, 
\[b''(x) \lesssim b'(x) |x|^{-1} \lesssim b'(x)
(b(x)+\varepsilon)^{-1/2m}\]
where the last step follows as $\varepsilon<0$ in this case.  Hence
\[\lambda^{-1}b''(x) (b(x)+\varepsilon)^{-3/2} \lesssim \lambda^{-1}
b'(x) (b(x)+\varepsilon)^{-(3m+1)/2m},\]
and the integral of the right side was evaluated in the previous
step, which completes the proof of \eqref{caseIVci_goal}.

{\em Case ii: } ($|b+\varepsilon| \le C \lambda^{-2/3} |\varepsilon|^{(2m-1)/3m} $).
We define
\[E[\phi] = \lambda^2 \lambda^{-\frac{1}{3}} |\varepsilon|^{\frac{2m-1}{6m}}
\phi^2 + \lambda^{\frac{1}{3}}|\varepsilon|^{-\frac{2m-1}{6m}} (\partial_x\phi)^2\]
and compute
\[\frac{dE}{dx} = 2\lambda^2 \Bigl(\lambda^{-\frac{1}{3}}
|\varepsilon|^{\frac{2m-1}{6m}}  -
  \lambda^{\frac{1}{3}}|\varepsilon|^{-\frac{2m-1}{6m}} 
    (b(x)+\varepsilon)\Bigr)\phi\partial_x \phi +2
    \lambda^{\frac{1}{3}}|\varepsilon|^{-\frac{2m-1}{6m}} \partial_x
      \phi g.\]
Thus,
\begin{equation}\label{dhalfe}\frac{d}{dx}E^{1/2} \lesssim
\lambda^{\frac{2}{3}}|\varepsilon|^{\frac{2m-1}{6m}} E^{1/2} +
\lambda^{1/6}|\varepsilon|^{-\frac{2m-1}{12m}} g.\end{equation}
An application of Gronwall's inequality to \eqref{dhalfe} yields
\eqref{caseIVc} in the given range when combined with \eqref{caseii}.

{\em Case iii: } ($b+\varepsilon \le -C \lambda^{-2/3} |\varepsilon|^{(2m-1)/3m}$)
We define $x_{\pm}$ so that $[x_-,x_+]=\{x\,:\, b+\varepsilon \le
-\alpha\}$.  We then multiply the equation \eqref{equation_tau_lambda} by
$-\lambda \phi$ and integrate over $[x_-,x_+]$ to obtain
\[\int_{x_-}^{x_+} \lambda (\partial_x\phi)^2 + \lambda^3
|b(x)+\varepsilon| \phi^2\,dx = \int_{x_-}^{x_+} -\lambda \phi g\,dx +
\lambda \phi \partial_x \phi\Bigl|_{x_-}^{x_+}.\]
But by the previous case, we have
\begin{equation}\label{bdytermsprevcase}\Bigl| \lambda \phi \partial_x \phi \Bigl|_{x_-}^{x_+}\Bigr|
\lesssim \|(\lambda^{-2/3} |\varepsilon|^{(2m-1)/3m}  +
|b(x)+\varepsilon|)^{-1/4}g\|^2_{L^1_x}.
\end{equation}
Thus, we have shown
\begin{multline}\label{intpiece}\int_{x_-}^{x_+} \lambda (\partial_x\phi)^2 + \lambda^3
|b(x)+\varepsilon| \phi^2\,dx \\\lesssim 
\lambda\|(\lambda^{-2/3} |\varepsilon|^{(2m-1)/3m}  +
|b(x)+\varepsilon|)^{1/4} \phi\|_{L^\infty}
\|(\lambda^{-2/3} |\varepsilon|^{(2m-1)/3m}  +
|b(x)+\varepsilon|)^{-1/4}g\|_{L^1_x}\\+
\|(\lambda^{-2/3} |\varepsilon|^{(2m-1)/3m}  +
|b(x)+\varepsilon|)^{-1/4}g\|^2_{L^1_x}.
\end{multline}
By the Fundamental Theorem of Calculus and \eqref{bdytermsprevcase},
we have
\begin{multline*}
  \lambda^2 (\lambda^{-2/3} |\varepsilon|^{(2m-1)/3m}  +
|b(x)+\varepsilon|)^{1/2} \phi^2(x) \lesssim
\|(\lambda^{-2/3} |\varepsilon|^{(2m-1)/3m}  +
|b(x)+\varepsilon|)^{-1/4}g\|^2_{L^1_x}
\\+ \int_{x_-}^{x_+} \lambda^2 (\lambda^{-2/3} |\varepsilon|^{(2m-1)/3m}  +
|b(x)+\varepsilon|)^{-1/2} |b'(x)| \phi^2 \\+ 2\lambda^2  (\lambda^{-2/3} |\varepsilon|^{(2m-1)/3m}  +
|b(x)+\varepsilon|)^{1/2} |\phi \partial_x\phi|\,dx.
\end{multline*}
Using \eqref{coercive}, it follows that this is
\[
  \lesssim \|(\lambda^{-2/3} |\varepsilon|^{(2m-1)/3m}  +
|b(x)+\varepsilon|)^{-1/4}g\|^2_{L^1_x}
+ \int_{x_-}^{x_+} \lambda^3 |b(x)+\varepsilon|\phi^2 + \lambda
(\partial_x \phi)^2\,dx.
\]
And thus, combined with \eqref{intpiece}, this completes the proof of
\eqref{caseIVc}.
\end{proof}

\section{Sharpness via quasimodes - Proof of Theorem \ref{thm_quasi}}

The notation here is much simpler if we work with complex valued
solutions to the wave equation.  Upon constructing a complex valued
solution that saturates the estimate, it immediately follows that
either the real part or the imaginary part of the constructed solution
is a real valued solution that also saturates the estimate.  Moreover,
we shall work within the $\phi=\pi/2$ plane, which is a restriction
that is preserved by our equation.

The proof is based on the following quasimode that was constructed in
\cite{CW}.
\begin{lem}[\cite{CW}]
  Given $\lambda\gg 1$ and $\alpha, \beta \in \R$, there exists a function $\tilde{u}\in C^3(\R)$ so that
  $\supp \tilde{u}\subseteq [-2 \lambda^{-\frac{1}{m+1}}, 2\lambda^{-\frac{1}{m+1}}]$,
  $\|\tilde{u}\|^2_{L^2}\approx \lambda^{\frac{m-1}{m+1}}$, and
\[-\lambda^{-2} \partial_x^2 \tilde{u} = \Bigl(E+\frac{x^{2m}}{m}\Bigr)\tilde{u}
- R,\]
where $E=(\alpha+i\beta) \lambda^{-\frac{2m}{m+1}}$ and $\|R\|_{L^2}
\le C \lambda^{-\frac{2m}{m+1}} \|\tilde{u}\|_{L^2}$. 
\end{lem}

We shall let $\psi$ solve
\[
\begin{cases}
  \Box_\g \psi = 0,\\
\psi(0,x,\theta,\phi) =
e^{i\lambda\theta}a(x)^{-1} \tilde{u}(x),\quad \partial_t\psi(0,x,\theta,\phi) =
i\tau e^{i\lambda\theta}a(x)^{-1} \tilde{u}(x).
\end{cases}
\]
To show that $\psi$ saturates the estimate, we construct the
approximate solution $v = e^{i\tau t} e^{i\lambda \theta} a(x)^{-1}
\tilde{u}(x)$.  Here $\tau$ satisfies $\tau^2 = \lambda^2 (1+E)$ with
$\Im \tau < 0$.  We compute that
\[\Box_\g v = a(x)^{-1} e^{i\tau t} e^{i\lambda \theta}\Bigl[
\Bigl(\lambda^2 -\lambda^2 \frac{x^{2m}}{m} - \lambda^2 a(x)^{-2}
-\frac{a''(x)}{a(x)}\Bigr)\tilde{u} + \lambda^2 R\Bigr]\]
with $(v, \partial_tv)|_{t=0} = (\psi, \partial_t
\psi)|_{t=0}$.  On the support of $\tilde{u}$, we have $a(x)^{-2} = 1
-\frac{x^{2m}}{m} + \O(\lambda^{-\frac{4m}{m+1}})$ and $\lambda^{-2} a''(x)/a(x) =
\O(\lambda^{-\frac{4m}{m+1}})$.  Thus, $\Box_\g v =
\tilde{R}$ where $\tilde{R} = \lambda^2 a(x)^{-1} e^{i\tau
  t}e^{i\lambda\theta} (R + \O(\lambda^{-\frac{4m}{m+1}})\tilde{u})$,
and in particular
\begin{equation}\label{tildeR}\|\tilde{R}\|_{L^2} \lesssim e^{|\Im \tau| t}
\lambda^{\frac{2}{m+1}} \|\tilde{u}\|_{L^2}.
\end{equation}

For $T=\varepsilon \lambda^{\frac{m-1}{m+1}}$ where $\varepsilon>0$
will be chosen later but is independent of $\lambda$, we compute
\[\|\partial_\theta v\|^2_{L^2L^2} = \int_0^T
\|e^{i\tau t} \partial_\theta \psi(0,\cd)\|^2_{L^2}\,dt = \int_0^T e^{2|\Im \tau| t}
\|\partial_\theta \psi(0,\cd)\|^2_{L^2}\,dt = T\cdot B(2T)
\|\partial_\theta \psi(0,\cd)\|^2_{L^2}
\]
where $B(T) = \frac{e^{T|\Im\tau|} - 1}{T|\Im\tau|}$.  Hence
\begin{equation}
  \label{approx_soln}
  \|\partial_\theta v\|^2_{L^2 L^2} = \varepsilon
  B(2T) \lambda^{\frac{m-1}{m+1}} \|\partial_\theta\psi(0,\cd)\|^2_{L^2}.
\end{equation}

We now estimate the error $w(t,x,\theta) = \psi(t,x,\theta) -
v(t,x,\theta)$, which solves
\[\Box_\g w = \tilde{R}, \quad (w,\partial_t w)|_{t=0} = (0,0).\]
Thus, by \eqref{energy},
\[\|(a(x))^{-1}\partial_\theta w\|_{L^\infty L^2} \le \int_0^T \|\tilde{R}\|_{L^2}\,ds
\lesssim \int_0^T e^{|\Im \tau|t} \lambda^{-\frac{m-1}{m+1}}
\|\partial_\theta \psi(0,\cd)\|_{L^2}\,ds = T\cdot B(T)
\lambda^{-\frac{m-1}{m+1}} \|\partial_\theta\psi(0,\cd)\|_{L^2}.\]
And hence
\begin{equation}\label{error_quasimode}\|(a(x))^{-1}\partial_\theta
  w\|_{L^2 L^2} \lesssim B(T)
T^{3/2}\lambda^{-\frac{m-1}{m+1}}\|\partial_\theta\psi(0,\cd)\|_{L^2}
=\varepsilon^{3/2} B(T)
\lambda^{\frac{m-1}{2(m+1)}}\|\partial_\theta\psi(0,\cd)\|_{L^2}.
\end{equation}

Using \eqref{approx_soln}, \eqref{error_quasimode}, and that $v$ is
supported on $[-1/2,1/2]$ if $\lambda$ is sufficiently large, it follows 
\begin{align*}\|\beta(|x|) \partial_\theta \psi\|^2_{L^2L^2} &\ge
\frac{1}{2}\|\beta(|x|) \partial_\theta v\|^2_{L^2L^2} -
\|\beta(|x|) \partial_\theta w\|^2_{L^2L^2}\\ &\ge \frac{1}{2}\|\partial_\theta
v\|^2_{L^2L^2} - \|a(x)^{-1}\partial_\theta w\|^2_{L^2L^2} \ge
\varepsilon \Bigl(\frac{1}{2} B(2T) - C^2\varepsilon^2
(B(T))^2\Bigr)\lambda^{\frac{m-1}{m+1}} \|\partial_\theta \psi(0,\cd)\|^2_{L^2}. 
\end{align*}
If we choose $\varepsilon = \frac{1}{2C}$, we can compute
\[\frac{1}{2}B(2T) - C^2\varepsilon^2 (B(T))^2 = \frac{1}{2}
\sum_{k=2}^\infty \frac{(k-2)2^{k-2}+1}{(k-1)!} (|\Im \tau|
T)^{k-2}\ge \frac{1}{2},\]
which completes the proof.


\bibliography{deg_wave}


\end{document}